\documentclass{mfat}
\pagespan{245}{255}

\usepackage{mmap}
\usepackage[utf8]{inputenc}

\newtheorem{theo}{Theorem}
\newtheorem{cor}{Corollary}
\newtheorem{lm}{Lemma}

\newtheorem{pr}{Proposition}

\begin{document}

\title[On  approximation of solutions of operator-differential
     equations $\ldots $]
    {On  approximation of solutions of operator-differential
     equations with their entire solutions of exponential type}


\author{V. M. Gorbachuk}
\address{National Technical University "KPI", 37 Peremogy Prosp., Kyiv, 06256, Ukraine}
\email{v.m.horbach@gmail.com}

\date{01/04/2016;\ \  Revised 12/04/2016}

\begin{abstract}
  We consider an equation of the form $y'(t) + Ay(t) = 0, \ t \in [0,
  \infty)$, where $A$ is a nonnegative self-adjoint operator in a
  Hilbert space. We give direct and inverse theorems on approximation
  of solutions of this equation with its entire solutions of
  exponential type. This establishes a one-to-one correspondence
  between the order of convergence to $0$ of the best approximation of
  a solution and its smoothness degree. The results are illustrated with
  an example, where the operator $A$ is generated by a second
  order elliptic differential expression in the space $L_{2}(\Omega)$
  \ (the domain $\Omega \subset \mathbb{R}^{n}$ is bounded with smooth
  boundary) and a certain boundary condition.
\end{abstract}

\maketitle

\vspace*{-5mm}

{\bf 1.} \ Let $A$ be a nonnegative self-adjoint operator in a Hilbert
space $\mathfrak{H}$ with a scalar product $(\cdot, \cdot)$. Denote by
$C_{\{1\}}(A)$ the set of all its exponential type entire vectors (see
[4]), namely,
$$
\begin{aligned}
C_{\{1\}}(A) = \Big\{f \in C^{\infty}(A) =  \bigcap_{n = 1}^{\infty}\mathcal{D}(A^{n})\big|
\exists \alpha  > 0, \ \exists c & = c(f) > 0:
\\
 \|A^{n}f\| & \leq c\alpha^{n},
\ n \in \mathbb{N}_{0} = \mathbb{N}\cup \{0\}\Big\}
\end{aligned}
$$
(everywhere in the sequel $c$ denotes various numerical constants
corresponding to the situations under consideration, $\mathcal{D}(A)$ is
a domain of $A$ and $\|f\| = \sqrt{(f, f)}$). The number
$$
\sigma (f, A) = \inf \left\{\alpha > 0 \bigl| \exists c > 0, \ \forall n \in \mathbb{N}_{0}: \|A^{n}f\| \leq c\alpha^{n}\right\}
$$
is called the type of the vector $f$ with respect to the operator $A$.

As has been shown in [3] that
$$
C_{\{1\}}(A) = \left\{f \in \mathfrak{H}\bigl| f = E(\lambda)g, \ \forall \lambda > 0, \ \forall g \in \mathfrak{H}\right\},
$$
where $E(\lambda) = E([0, \lambda])$ is the spectral measure of $A$.

Now, consider the equation
$$
y'(t) + Ay(t) = 0, \quad t \in \mathbb{R}_{+} = [0, \infty). \leqno (1)
$$
By a weak solution of this equation we mean a continuous vector-valued
function $y(t):\mathbb{R}_{+} \mapsto \mathfrak{H}$ such that for any
$t \in \mathbb{R}_{+}$,
\vspace*{-1mm}
$$
\int_{0}^{t} y(s)\,ds \in \mathcal{D}(A) \quad \text{and} \quad y(t) = -A \int_{0}^{t} y(s)\,ds + y(0).
$$
\vspace*{-2mm}

Denote by $S$ the set of all weak solutions of (1). As it was established
in [1],
$$
S = \left\{y(t):\mathbb{R}_{+} \mapsto \mathfrak{H}\big| \ y(t) = e^{-At}f, f \in \mathfrak{H} \right\}, \leqno (2)
$$
where
$$
e^{-At}f = \int_{0}^{\infty} e^{-\lambda t}\,dE(\lambda)f.
$$
Note that the set of all strong (or simply) solutions of (1) is given
by formula (2) with $f$ ranging over whole $\mathcal{D}(A)$.

It is not difficult to verify that $S$ is a Hilbert space with the norm
$$
\|y\|_{S} = \sup_{t \in \mathbb{R}_{+}} \left\|e^{-At}f\right\| = \|f\|. \leqno (3)
$$

If the operator $A$ is bounded, then each weak solution $y(t)$ of
equation (1) can be extended to an entire $\mathfrak{H}$-valued vector
function $y(z)$ of exponential type,
$$
\sigma (y) = \inf\left\{\alpha > 0: \|y(z)\| \leq ce^{\alpha |z|}\right\}.
$$
But it is not the case if $A$ is unbounded. The set $S_{0}$ of all
weak solutions of (1) admitting an extension to an entire vector
function of exponential type is described by the following theorem.

\begin{theo} A weak solution $y(t)$ of equation (1) belongs to $S_{0}$
  if and only if it can be represented in form (2) with $f \in
  C_{\{1\}}(A)$. The set $S_{0}$ is dense in $S$, and $\sigma (y) =
  \sigma (f, A)$.
\end{theo}

\begin{proof}
  Let $f \in C_{\{1\}}(A)$. Then $f = E(\alpha)f$, where $\alpha =
  \sigma (f, A)$. By (2),
$$
y (t) = \int_{0}^{\alpha} e^{-\lambda t}\,dE(\lambda)f.
$$
From this it follows that $y(t)$ can be extended to an entire
vector-valued function $y(z)$ and
$$
\|y(z)\|^{2} = \int_{0}^{\alpha} e^{-2\text{Re}z\lambda}\,d(E(\lambda)f, f) \leq e^{2\sigma(f, A)|z|}\|f\|^{2},
$$
that is, $y(z)$ is an entire $\mathfrak{H}$-valued function of exponential type $\sigma (y) \leq \sigma (f, A)$.

Conversely, if $y(t) = e^{-At}f$ admits an extension to an entire
vector-valued function of exponential type $\sigma (y)$, then, by
virtue of
$$
\|y(-t)\|^{2} = \int_{0}^{\infty} e^{2\lambda t}\,d(E(\lambda)f, f) \leq ce^{2\sigma(y)t}, \quad t \geq 0,
$$
we have
$$
\int_{\sigma (y)}^{\infty} e^{2t(\lambda - \sigma (y))}\,d(E(\lambda)f, f) \leq c.
$$
Passing to the limit under the integral sign as $t \to \infty$, we
conclude, on the basis of the Fatou theorem, that the measure
generated by the monotone function $(E(\lambda)f, f)$ is concentrated
on the interval $[0, \sigma (y)]$. Hence, $\sigma (f, A) \leq \sigma
(y)$.

Density of $S_{0}$ in $S$ follows from the density in $\mathfrak{H}$
of the set $\{E([0, \alpha])f, \forall \alpha >0, \forall f \in
\mathfrak{H}\}$.
\end{proof}

In view of Theorem 1, it is reasonable to ask whether it is possible
to approximate an arbitrary weak solution of equation (1) with its
exponential type entire solutions. An answer to the question is given
below.  We prove direct an inverse theorems which ascertain the
relationship between the degree of smoothness of a solution and the
rate of convergence to $0$ of its best approximation. In doing so, the
operator approach developed in [7, 4, 5] plays an important role.

\medskip {\bf 2.} \ Recall some definitions and notations of the
approximation theory required to formulate further results.

For $y \in S $ and a number $r > 0$, we put
$$
\mathcal{E}_{r}(y) = \inf_{y_{0} \in S_{0}: \sigma (y_{0}) \leq r} \|y - y_{0}\|_{S}.
$$
Thus, $\mathcal{E}_{r}(y)$ is the best approximation of a weak
solution $y(t)$ of equation (1) with its entire solutions of exponential
type not exceeding $r$. If $y$ is fixed, the function
$\mathcal{E}_{r}(y)$ does not increase and, since $\overline{{S}_{0}}
= S, \ \mathcal{E}_{r}(y) \to 0$ as $r \to \infty$.

For $f \in \mathfrak{H}$, set also
$$
\mathcal{E}_{r}(f, A) = \inf_{f_{0} \in C_{\{1\}}(A): \sigma (f_{0}, A) \leq r} \|f - f_{0}\| = \|(I - E(r))f\|.
$$
If $y(t) = e^{-At}f$, then, by virtue of (3),
$$
\mathcal{E}_{r}(y) = \mathcal{E}_{r}(f, A). \leqno (4)
$$

Besides, for any $k \in \mathbb{N}_{0}$, we introduce the function
$$
\omega_{k}(t, y) = \sup_{|h| \leq t}\sup_{s \in \mathbb{R}_{+}} \Big\| \sum_{j = 0}^{k}
(-1)^{k - j}C_{k}^{j}y(s + jh)\Big\|, \quad k \in \mathbb{N}; \quad \omega_{0}(t, y) \equiv \|y\|_{S}, \quad t > 0.
$$
Taking into account (2) and the equality $e^{-As}y(t) = y (t + s)$, we conclude that
$$
\forall k \in \mathbb{N}_{0}:\omega_{k}(t, y) = \sup_{|h| \leq t}\left\|\left(e^{-Ah} - I\right)^{k}y\right\|_{S}
$$
($I$ is the identity operator).

The following theorem establishes a relation between
$\mathcal{E}_{r}(y)$ and $\omega_{k}(t, y)$, and it is an analog of
the well-known Jackson's theorem on  approximation of a continuous
periodic function by trigonometric polynomials.
\begin{theo} Let $y \in S$. Then
$$
\forall k \in \mathbb{N}, \ \exists c_{k} > 0: \mathcal{E}_{r}(y) \leq c_{k}\omega_{k}\left(\frac{1}{r}, y\right), \quad  r > 0. \leqno (5)
$$
\end{theo}

\begin{proof}
By (2), $y(t) = e^{-At}f, \ f \in \mathfrak{H}$. From (3), (4), it follows that
$$
\begin{aligned}
\omega_{k}^{2}(t, y)   = & \sup_{0 \leq s \leq t}  \left\|\left(e^{-As} - I\right)^{k}y\right\|_{S}^{2}  \geq \left\|\left(e^{-At} - I\right)^{k}y\right\|_{S}^{2} = \sup_{s \in \mathbb{R}_{+}} \left\|\left(e^{-At} - I\right)^{k}e^{-As}f\right\|^{2}
\\
 = & \left\|\left(e^{-At} - I\right)^{k}f\right\|^{2} = \int_{0}^{\infty}\left(e^{-\lambda t} - 1\right)^{2k}\, d(E(\lambda)f, f)
\\
 \geq & \int_{\frac{1}{t}}^{\infty}\left(e^{-\lambda t} - 1\right)^{2k}\, d(E(\lambda)f, f)
\geq \left(1 - e^{-1}\right)^{2k}\mathcal{E}_{\frac{1}{t}}(y).
\end{aligned}
$$
So,
$$
\forall t > 0: \mathcal{E}_{\frac{1}{t}}(y) \leq \left(1 - e^{-1}\right)^{k}\omega_{k}(t, y).
$$
Setting $r = \frac{1}{t}$ and $c_{k} = \left(1 - e^{-1}\right)^{k}$, we obtain (5).
\end{proof}


Denote by $C^{n}(\mathbb{R}_{+}, \mathfrak{H})$ the set of all $n$
times continuously differentiable on $\mathbb{R}_{+}$
$\mathfrak{H}$-valued vector-valued functions. Since the operator $A$
is closed, the inclusion $y \in S \cap C^{n}(\mathbb{R}_{+},
\mathfrak{H})$ implies that $y^{(k)} \in S, \ k = 1, 2, \dots, n$.
\begin{theo} Suppose that $y \in C^{n}(\mathbb{R}_{+}, \mathfrak{H}), \ n
  \in \mathbb{N}_{0}$. Then
$$
\forall r > 0, \ \forall k \in \mathbb{N}_{0}: \mathcal{E}_{r}(y) \leq \frac{c_{k + n}}{r^{n}}\omega_{k}\left(\frac{1}{r}, y^{(n)}\right),
$$
where the constants $c_{k}$ are the same as in Theorem 2.
\end{theo}

\begin{proof}
  Let $y \in C^{n}(\mathbb{R}_{+}, \mathfrak{H}), r > 0$ and $0 \leq t
  < \frac{1}{r}$. Using properties of a contraction $C_{0}$-semigroup,
  we get
$$
\left\|\left(e^{-At} - I\right)^{k + n}y(s)\right\| = \left\|\left(e^{-At} - I\right)^{n}\left(e^{-At} - I\right)^{k}y(s)\right\|
$$
\vspace*{-2mm}
$$
\leq \int_{0}^{t} \dots \int_{0}^{t} \left\|e^{-A(s_{1} + \dots s_{n})}\right\|\left\|\left(e^{-At} - I\right)^{k}A^{n}y(s)\right\|\,ds_{1} \dots ds_{n}
$$
\vspace*{-2mm}
$$
 \leq t^{n}\left\|\left(e^{-At} - I\right)^{k}y^{(n)}(s)\right\|,
$$
whence
$$
\omega_{k + n}\left(\frac{1}{r}, y\right) \leq \frac{1}{r^{n}} \omega_{k}\left(\frac{1}{r}, y^{(n)}\right)
$$
and, because of (5),
$$
\mathcal{E}_{r}(y) \leq c_{k + n}\omega_{k + n}\left(\frac{1}{r}, y\right) \leq \frac{c_{k + n}}{r^{n}}\omega_{k}\left(\frac{1}{r}, y^{(n)}\right),
$$
which is what had to be proved.
\end{proof}


Setting, in Theorem 3, $k = 0$ and taking into account that
$\omega_{0}(t, y^{(n)}) = \|y^{(n)}\|_{S}$, we arrive at the following
assertion.

\begin{cor} Let $y \in C^{n}(\mathbb{R}_{+}, \mathfrak{H}), \ n \in
  \mathbb{N}$. Then
$$
\forall r > 0: \mathcal{E}_{r}(y) \leq \frac{c_{n}}{r^{n}}\|y^{(n)}\|_{S}.
$$
\end{cor}

For numbers $h > 0$ and $k \in \mathbb{N}_{0}$, we put
$$
\Delta_{h}^{k} = \left(e^{-Ah} - I\right)^{k} = \sum_{j = 0}^{k}(-1)^{k - j}C_{k}^{j}e^{-Ajh}.
$$
\begin{lm} If $y \in S_{0}$ and $\sigma (y) = \alpha$, then
$$
\forall h > 0, \ \forall k, n \in \mathbb{N}_{0}: \left\|\Delta_{h}^{k}y^{(n)}\right\|_{S} \leq (\alpha h)^{k}\alpha^{n}\|y\|_{S}. \leqno (6)
$$
\end{lm}

\noindent{\it Proof}.
  It follows from the inequality
$$
1 - \lambda h - e^{-\lambda h} \leq 0 \quad  (\lambda \geq 0, \ h > 0)
$$
and the representation $y(t) = e^{-At}f$ that
$$
\begin{aligned}
\left\|\Delta_{h}^{k}y^{(n)}\right\|^{2} = &
\int_{0}^{\alpha} \left(1 - e^{-\lambda h}\right)^{2k}e^{-2\lambda
t}\lambda^{2n}\,d(E(\lambda)f,f)
\\
\leq & \int_{0}^{\alpha}(\lambda h)^{2k}\lambda^{2n}\,d(E(\lambda)f, f) \leq (\alpha h)^{2k}\alpha^{2n}\|f\|.
\end{aligned}
$$
\vspace{-1mm}
This and (3) imply
$$
\qquad\qquad\qquad\qquad\quad\left\|\Delta_{h}^{k}y^{(n)}\right\|_{S} \leq (\alpha h)^{k}\alpha^{n}\|f\| = (\alpha h)^{k}\alpha^{n}\|y\|_{S}.\qquad\qquad\qquad\qquad\quad\qed
$$
\vspace{-2mm}


Taking in (6) $k = 0$, we arrive at an analog of Bernstein's
inequality, namely \vspace*{-1mm}
$$
\forall n \in \mathbb{N}: \left\|y^{(n)}\right\| \leq \alpha^{n}\|y\|_{S}. \leqno (7)
$$
Putting there $n = 0$, we obtain
$$
\forall n \in \mathbb{N}: \left\|\Delta_{h}^{k}y\right\|_{S} \leq (\alpha h)^{k}\|y\|_{S} = (\alpha h)^{k}\alpha^{n}\|y\|_{S}.
$$

It should be noted that the inequality
$$
\mathcal{E}_{r}(y) \leq \frac{c}{r^{n}}, \quad r > 0, \quad  n \in \mathbb{N},     \leqno (8)
$$
does not yet imply the inclusion $y \in C^{n}(\mathbb{R}_{+},
\mathfrak{H})$. Nevertheless, the following statement, inverse to
Theorem 3, is valid.

\begin{theo} Suppose that $y \in S$, and let $\omega (t)$ be a continuity
  module type function, i.e.,

  $1) \ \omega (t)$ is continuous and nondecreasing on
  $\mathbb{R}_{+}$;

$2) \ \omega (0) = 0$;

$3) \ \exists c > 0, \ \forall t > 0: \omega (2t) \leq c\omega (t)$. \\
In order that $y \in C^{n}(\mathbb{R}_{+}, \mathfrak{H})$, it is sufficient that there exist a number $m > 0$ such that \vspace*{-1mm}
$$
\forall r > 0, \ \forall n \in \mathbb{N}: \mathcal{E}_{r}(y) \leq \frac{m}{r^{n}}\omega \left(\frac{1}{r}\right). \leqno (9)
$$
\end{theo}

\begin{proof}
  Assume that, for $y \in S$, condition (9) is fulfilled. Then there
  exists a sequence $y_{i} \in S: \sigma (y_{i}) < 2^{i} \ (i \in
  \mathbb{N})$ such that
$$
\|y - y_{i}\|_{S} \to 0 \quad \text{as} \quad i \to \infty.
$$
In view of (7) and the inequality $ \sigma (y_{i} - y_{i - 1}) < 2^{i}
\ (i \in \mathbb{N})$, we have
$$
\begin{aligned}
\left\|y_{i}^{(n)} - y_{i - 1}^{(n)}\right\|_{S} \leq & 2^{in}\|y_{i} - y_{i - 1}\|_{S} \leq 2^{in}(\|y - y_{i} \|_{S} + \|y - y_{i - 1}\|_{S})
\\
\leq & 2^{in}\left(\frac{m}{2^{in}}\omega\left(\frac{1}{2^{i}}\right) +
\frac{m}{2^{(i - 1)n}}\omega\left(\frac{1}{2^{i - 1}}\right)\right).
\end{aligned}
$$
From this it follows that
$$
\left\|y_{i}^{(n)} - y_{i - 1}^{(n)}\right\|_{S} \leq m2^{n}\left(\frac{1}{2^{n}} + c\right)\omega\left(\frac{1}{2^{i}}\right),
$$
and therefore
$$
\left\|y_{i}^{(n)} - y_{i - 1}^{(n)}\right\|_{S} \to 0 \quad \text{as} \quad i \to \infty.
$$
Since the space $S$ is complete, there exists $\widetilde{y} \in S$ such that
$$
\left\|y_{i}^{(n)} - \widetilde{y}\right\|_{S} \to 0 \quad \text{when} \quad i \to \infty.
$$
Thus, $y_{i} \to y, \ y_{i}^{(n)} \to \widetilde{y} \ (i \to \infty)$
in the space S.  Taking into account that the operator
$\frac{d^{n}}{dt^{n}}$ is closed in $S$, we conclude that $y \in
C^{n}(\mathbb{R}_{+}, \mathfrak{H})$ and $y^{(n)}(t) \equiv
\widetilde{y}(t)$.
\end{proof}

Replacing in inequality (8) $n$ by $n + \varepsilon$ and thus
strengthening it, we shall arrive at the following consequence.

\begin{cor} Let, for $y \in S$,
$$
\exists c > 0, \ \exists \varepsilon > 0: \mathcal{E}_{r}(y) \leq \frac{c}{r^{n + \varepsilon}}.
$$
Then $y \in C^{n}(\mathbb{R}_{+}, \mathfrak{H})$.
\end{cor}

\vspace{-2mm}
\medskip {\bf 3.} \ Now let $\{m_{n}\}_{n \in \mathbb{N}_{0}}$ be a
nondecreasing sequence of numbers (there is no loss of generality in
assuming that $m_{0} = 1$). We put
$$
C_{\{m_{n}\}} = C_{\{m_{n}\}}(\mathbb{R}_{+}, \mathfrak{H}) = \bigcup_{\alpha > 0}C_{m_{n}}^{\alpha},
\quad C_{(m_{n})} = C_{(m_{n})}(\mathbb{R}_{+}, \mathfrak{H}) = \bigcap_{\alpha > 0}C_{m_{n}}^{\alpha},
$$
where
$$
\begin{aligned}
C_{m_{n}}^{\alpha} & =  C_{m_{n}}^{\alpha}(\mathbb{R}_{+}, \mathfrak{H})
\\
& = \Big \{y \in C^{\infty}(\mathbb{R}_{+}, \mathfrak{H})\bigl| \exists c = c(y) > 0,
\ \forall k \in \mathbb{N}_{0}:\sup_{t \in \mathbb{R}_{+}}
\Big\|y^{(k)}(t)\Big\| \leq cm_{k}\alpha^{k}\Big\}
\end{aligned}
$$
\vspace{-2mm}
is a Banach space with respect to the norm
$$
\|y\|_{C_{m_{n}}^{\alpha}} = \sup_{k \in \mathbb{N}}\frac{\sup_{t \in \mathbb{R}_{+}}\left\|y^{(k)}(t)\right\|}{\alpha^{k}m_{k}}.
$$
The spaces $C_{\{m_{n}\}}$ and $C_{(m_{n})}$ are equipped with the
topologies of inductive and projective limits of the spaces
$C_{m_{n}}^{\alpha}$, respectively. Note that the spaces $C_{\{n!\}}, \
C_{(n!)} \ (m_{n} = n!)$ and $C_{\{1\}} \ (m_{n} \equiv 1)$ are
nothing that, respectively, the spaces of bounded on $\mathbb{R}_{+}$
with all their derivatives analytic, entire, and entire of exponential
type $\mathfrak{H}$-valued vector functions.

In what follows, we assume in addition that $\{m_{n}\}_{n \in
  \mathbb{N}_{0}}$ satisfies the condition
$$
\forall \alpha > 0, \ \exists c = c(\alpha): m_{n} \geq c\alpha^{n} \leqno (10)
$$
and put
\vspace{-3mm}
$$
\tau (\lambda) = \sum_{n = 0}^{\infty} \frac{\lambda^{n}}{m_{n}}. \leqno (11)
$$
It is clear that $\tau (\lambda)$ is entire, $\tau (\lambda) \geq 1$ for $\lambda \geq 0$, and $\tau (\lambda) \uparrow \infty$ as $\lambda \to \infty$.

\begin{theo} Suppose the condition
$$
\exists c > 0, \ \exists h > 1, \ \forall n \in \mathbb{N}_{0}: m_{n + 1} \leq ch^{n}m_{n}  \leqno (12)
$$
to be fulfilled for the sequence $\{m_{n}\}_{n \in
  \mathbb{N}_{0}}$. Then the following equivalence relations hold:
$$
\begin{array}{rcl}
y \in C^{\infty}(\mathbb{R}_{+}, \mathfrak{H}) & \Longleftrightarrow & \forall \alpha > 0:
\mathcal{E}_{r}(y) = O\left(\frac{1}{r^{\alpha}}\right) \quad (r \to \infty), \\
y \in C_{\{m_{n}\}} & \Longleftrightarrow & \exists \alpha > 0:
\mathcal{E}_{r}(y) = O\left(\tau^{-1}(\alpha r)\right) \quad (r \to \infty), \\
y \in C_{(m_{n})} & \Longleftrightarrow & \forall \alpha > 0:
\mathcal{E}_{r}(y) = O\left(\tau^{-1}(\alpha r)\right) \quad (r \to \infty) \\
\end{array}
$$
$(\tau (\lambda)$ is defined by (11)).
\end{theo}

\begin{proof}
  Let $C^{\infty}(A)$ denote the set of all infinitely differentiable
  vectors of the operator $A$,
$$
C^{\infty}(A) = \bigcap_{n \in \mathbb{N}_{0}} \mathcal{D}(A^{n}).
$$
For a number $\alpha > 0$, we put
$$
C_{m_{n}}^{\alpha}(A) = \left\{f \in C^{\infty}(A)\bigl| \exists c = c(f) > 0, \ \forall n \in \mathbb{N}_{0}: \left\|A^{n}f\right\| \leq c\alpha^{n}m_{n}\right\}.
$$
The set $C_{m_{n}}^{\alpha}(A)$ is a Banach space with respect to the norm
$$
\|f\|_{C_{m_{n}}^{\alpha}(A)} = \sup_{n \in \mathbb{N}_{0}} \frac{\left\|A^{n}f\right\|}{\alpha^{n}m_{n}}.
$$
Then
$$
C_{\{m_{n}\}}(A) = \bigcup_{\alpha > 0}C_{m_{n}}^{\alpha}(A) \ \text{and} \ C_{(m_{n})}(A) = \bigcap_{\alpha > 0}C_{m_{n}}^{\alpha}(A)
$$
are linear locally convex spaces with topologies of the inductive and
the projective limits, respectively.

Let
$$
\mathcal{E}_{r}(f, A) = \inf_{f_{0} \in C_{\{1\}}(A)} \|f - f_{0}\|.
$$
As it has been shown in [4], the following equivalence relations take place:
$$
\begin{array}{rcl}
f \in C^{\infty}(A) & \Longleftrightarrow & \forall \alpha > 0:
\mathcal{E}_{r}(f, A) = O\left(\frac{1}{r^{\alpha}}\right) \quad (r \to \infty), \\
f \in C_{\{m_{n}\}}(A) & \Longleftrightarrow & \exists \alpha > 0:
\mathcal{E}_{r}(f, A) = O\left(\tau^{-1}(\alpha r)\right) \quad (r \to \infty), \\
f \in C_{(m_{n})}(A) & \Longleftrightarrow & \forall \alpha > 0:
\mathcal{E}_{r}(f, A) = O\left(\tau^{-1}(\alpha r)\right) \quad  (r \to \infty). \\
\end{array} \leqno (13)
$$

Consider the map $F: \mathfrak{H} \mapsto S$, \vspace*{-1mm}
$$
Ff = e^{-At}f.
$$
Since, for $y \in S$, there exists a unique vector $f \in \mathfrak{H}$
such that $y = e^{-At}f$, this transformation is one-to-one. From (3)
it follows that $F$ maps $\mathfrak{H}$ onto $S$ isometrically and
$$
F\left(C^{\infty}(A)\right) =  C^{\infty}(\mathbb{R}_{+}, \mathfrak{H}), \quad F\left(C_{\{m_{n}\}}(A)\right) = C_{\{m_{n}\}}, \quad F\left(C_{(m_{n})}(A)\right) = C_{(m_{n})}. \leqno (14)
$$
The proof of the theorem follows from (13), (14) because of
$\mathcal{E}_{r}(f, A) = \mathcal{E}_{r}\left(e^{-At}f\right)$.
\end{proof}

If $m_{n} = n^{n\beta} \ (\beta > 0)$, then $\tau (r) =
e^{-r^{1/\beta}}$ and Theorem 5 yields the following assertion.

\begin{cor} The following equivalence relations are valid:
$$
\begin{array}{rcl}
y \in C_{\{n^{n\beta}\}}(A) & \Longleftrightarrow & \exists \alpha > 0:
\mathcal{E}_{r}(y) = O\left(e^{-\alpha r^{1/\beta}}\right) \quad (r \to \infty), \\
y \in C_{(n^{n\beta})}(A) & \Longleftrightarrow & \forall \alpha > 0:
\mathcal{E}_{r}(y) = O\left(e^{-\alpha r^{1/\beta}}\right) \quad  (r \to \infty). \\
\end{array}
$$
\end{cor}

Recall that an entire $\mathfrak{H}$-valued vector function
$x(\lambda)$ has a finite order of growth if
$$
\exists \gamma > 0, \ \forall \lambda \in \mathbb{C}: \|x(\lambda)\| \leq \exp(|\lambda|^{\gamma}).
$$
The greatest lower bound $\rho (x)$ of such $\gamma$ is the order of
$x(\lambda)$. The type of an entire vector-valued function
$x(\lambda)$ of an order $\rho$ is determined as
$$
\sigma (x) = \inf \left\{a > 0: \|x(\lambda)\| \leq \exp(a|\lambda|^{\rho})\right\}.
$$
Since the semigroup $\left\{e^{-At}\right\}_{t \geq 0}$ is analytic,
every weak solution $y(t)$ of equation (1) is analytic on $(0,
\infty)$. It is not difficult to show that it is analytic on $[0,
\infty)$ if and only if $y \in C_{\{n^{n}\}}$.

By Corollary 3,
$$
\exists \alpha > 0: \mathcal{E}_{r}(y) = O\left(e^{-\alpha r}\right) \quad (r \to \infty).
$$
As for the extendability of $y(t)$ to an entire vector function of
order $\rho$ and finite type, the answer to the question gives the
next theorem.

\begin{theo} In order that a weak solution of equation (1) admit an
  extension to an entire $\mathfrak{H}$-valued vector function $y(z)$,
  it is necessary and sufficient that
$$
\forall \alpha > 0: \mathcal{E}_{r}(y) = O\left(e^{-\alpha r}\right) \quad (r \to \infty).  \leqno (15)
$$
The extension $y(z)$ is of finite order $\rho$ and finite type if and only if
\vspace*{-1mm}
$$
\exists \alpha > 0: \mathcal{E}_{r}(y) = O\left(e^{-\alpha r^{1/\beta}}\right) \quad  (r \to \infty),
$$\vspace*{-1mm}
where $\beta$ and $\rho$ are connected with each other by the formula
\vspace*{-0mm}
$$
\beta = \frac{\rho - 1}{\rho} < 1.
$$
\end{theo}

Note that we may always suppose $\rho > 1$. Indeed, if $\rho \leq 1$
and the type is finite, then $y \in S_{0}$ and it has no sense to
approximate a solution from $S_{0}$ by solutions from the same space.

\vspace*{1mm}
\noindent{\it Proof of Theorem 6.}
  Let $y(t)$ be a weak solution of equation (1). By Corollary 3, $y
  \in C_{(n^{n})}$, so $y(t)$ admits an extension to an entire vector
  function if and only if relation (15) is fulfilled.

  Assume that $y(t)$ admits an extension to an entire
  $\mathfrak{H}$-valued vector function $y(z)$ which has an order
  $\rho$ and a finite type $\sigma$. Then
$$
\forall \sigma_{1} > \sigma, \ \exists c = c(\sigma_{1}): \|y(z)\| \leq ce^{\sigma_{1}|z|^{\rho}}.
$$
Hence, \vspace*{-2mm}
$$
\forall r > 0, \ \forall n \in \mathbb{N}_{0}: \left\|y^{(n)}(z)\right\| \leq \frac{n!}{2\pi} \int_{|z - \zeta| = r} \frac{\|y(\zeta)\|}{|z - \zeta|^{n + 1}}\,d\zeta \leq \frac{c\sigma_{1}n!}{r^{n}}\exp\left(2\sigma_{1}r^{\rho}\right).
$$
Taking into account that the function
$\frac{\exp\left(ar^{\rho}\right)}{r^{n}}$ reaches its minimum at the
point $\left(\frac{n}{a\rho}\right)^{1/\rho}$ and using Stirling's
formula
$$
n! = n^{n}e^{-n}\sqrt{2\pi n}\left(1 + O\left(\frac{1}{n}\right)\right) \quad (n \to \infty),
$$
we get \vspace*{-1mm}
$$
\left\|y^{(n)}(z)\right\| \leq c\left(2e^{1 - \rho}\sigma_{1}\rho\right)^{\frac{1}{\rho}}n^{\frac{\rho - 1}{\rho}n},
$$
which shows that $y \in C_{\{n^{n\beta}\}}$, where
$$
\beta \leq \frac{\rho - 1}{\rho} \Longleftrightarrow \rho \geq \frac{1}{1 - \beta}, \quad \beta < 1. \leqno (16)
$$
By Corollary 3,
$$
\exists \alpha > 0: \mathcal{E}_{r}(y) = O\left(e^{-\alpha r^{1/\beta}}\right) \quad (r \to \infty). \leqno (17)
$$

Conversely, let (17) hold true. Then Corollary 3 implies that $y \in
C_{\{n^{n\beta}\}} \ (0 < \beta < 1)$ is an entire vector-valued
function and it can be represented by the series $\sum_{k =
  0}^{\infty}\frac{y^{(k)}(0)}{k!}z^{k}$. For its order of growth
$\rho$ we have
$$
\rho = \varlimsup_{n \to \infty} \frac{n\ln n}{\ln \frac{n!}{\left\|y^{(n)}(0)\right\|}} \leq
\varlimsup_{n \to \infty} \frac{n\ln n}{\ln \left(n!n^{-n\beta}\right)} = \frac{1}{1 - \beta},
$$
that is,
$$
\rho \leq \frac{1}{1 - \beta}. \leqno (18)
$$
It follows from (16) and (18) that
$$
\qquad\qquad\qquad\qquad\qquad\qquad\quad\rho = \frac{1}{1 - \beta}, \quad 0 < \beta < 1.\qquad\qquad\qquad\qquad\qquad\qquad\quad \qed
$$

\medskip {\bf 4.} \ The direct and inverse theorems of the
approximation theory are usually formulated for Banach space. Proving
them is slightly more complex than in the case of a Hilbert
space. Below we show how, for example, Theorem 5 can be reformulated
in a Banach space. To this end we introduce the following notations:
$$
\mathfrak{H}^{n} = \mathcal{D}(A^{n}), \|f\|_{\mathfrak{H}^{n}} = \left(\|f\|^{2} + \|A^{n}f\|^{2}\right)^{1/2}.
$$
The space $\mathfrak{H}^{n}$ is continuously and densely embedded into
$\mathfrak{H}$. Denote by $\mathfrak{H}^{-n}$ the completion of
$\mathfrak{H}$ in the norm $\|f\|_{\mathfrak{H}^{-n}} = \left\|\left(A
    + I\right)^{-n}f\right\|$, the so called negative space associated
with the positive space $\mathfrak{H}^{n}$ in the chain
$$
\mathfrak{H}^{n} \subseteq \mathfrak{H} \subseteq \mathfrak{H}^{-n}
$$
(see [2]). By a suitable choice of the norms in $\mathfrak{H}^{n}$ and $\mathfrak{H}^{-n}$, one can attain
the relation
$$
\|f\|_{\mathfrak{H}^{-n}} \leq \|f\| \leq \|f\|_{\mathfrak{H}^{n}}.
$$

\begin{theo} Let $\mathfrak{B}$ be a Banach space inside the chain
$$
\mathfrak{H}^{k_{1}} \subseteq \mathfrak{B} \subseteq \mathfrak{H}^{-k_{2}}  \leqno (19)
$$
of continuously and densely embedded into each other spaces with some
$k_{1}, k_{2} \in \mathbb{N}$, and let a sequence $\{m_{n}\}_{n \in
  \mathbb{N}_{0}}$ possess the properties (10) and (12). Then for a
weak solution $y(t)$ of equation (1), the following equivalence
relations hold true:
$$
\begin{array}{rcl}
y \in C^{\infty}(\mathbb{R}_{+}, \mathfrak{H}) & \Longleftrightarrow & \forall \alpha > 0:
\mathcal{E}_{r}(y, \mathfrak{B}) = O\left(r^{-\alpha}\right) \quad  (r \to \infty), \\
y \in C_{\{m_{n}\}} & \Longleftrightarrow & \exists \alpha > 0,:
\mathcal{E}_{r}(y, \mathfrak{B}) = O\left(\tau^{-1}(\alpha r)\right) \quad  (r \to \infty), \\
y \in C_{(m_{n})} & \Longleftrightarrow & \forall \alpha > 0:
\mathcal{E}_{r}(y, \mathfrak{B}) = O\left(\tau^{-1}(\alpha r)\right) \quad (r \to \infty),\\
\end{array}
$$
where
$$
\mathcal{E}_{r}(y, \mathfrak{B}) = \inf_{y_{0} \in S_{0}: \sigma (y_{0}) \leq r} \sup_{s \in \mathbb{R}_{+}} \|y(s) - y_{0}(s)\|_{\mathfrak{B}},
$$
$\tau (\lambda)$ is defined by (11), $\|\cdot\|_{\mathfrak{B}}$ is the norm in $\mathfrak{B}$.
\end{theo}

\begin{proof}
  Show first that the spaces $C_{\{m_{n}\}}(A)$ and $C_{(m_{n})}(A)$
  considered as subspaces of $\mathfrak{H}$ coincide with the
  corresponding subspaces $C_{\{m_{n}\}}^{k}(A)$ and
  $C_{(m_{n})}^{k}(A)$ constructed in the Hilbert space
  $\mathfrak{H}^{k}$ from the restriction $A\upharpoonright
  \mathfrak{H}^{k}$ which is a nonnegative self-adjoint operator
  in~$\mathfrak{H}^{k}$.

So, let $f \in C_{\{m_{n}\}}(A)$. Then

$$
\exists \alpha > 0, \ \exists c > 0: \left\|A^{i}f\right\|_{\mathfrak{H}^{k}} = \left(\left\|A^{i}f\right\|^{2} + \left\|A^{i + k}f\right\|^{2}\right)^{1/2}
$$
$$
\qquad \qquad \qquad \leq c\left(\alpha^{2i}m_{i}^{2} + \alpha^{2(i +
    k)}m_{k + i}^{2}\right)^{1/2} \leq \widetilde{c}\left(\alpha
  h^{k}\right)^{i}m_{i}, \leqno (20)
$$
i.e., $ C_{\{m_{n}\}}(A) \subseteq C_{\{m_{n}\}}^{k}(A)$. From (19) we
also have the embedding $ C_{(m_{n})}(A) \subseteq
C_{(m_{n})}^{k}(A)$. The inverse embeddings are consequences of the
estimate $\left\|A^{i}f\right\|_{\mathfrak{B}} \leq
\left\|A^{i}f\right\|_{\mathfrak{H}^{k}}$. Thus we have
$$
C_{\{m_{n}\}}(A) = C_{\{m_{n}\}}^{k}(A), \quad C_{(m_{n})}(A) = C_{(m_{n})}^{k}(A). \leqno (21)
$$
It is also evident that $C_{k}^{\infty}(A) = C^{\infty}(A)$.

Since for a vector $g \in C_{\{1\}}(A) \ (m_{n} \equiv 1)$ of type
$\sigma(g, A) \leq k$, the inequality
$$
\left\|A^{i}f\right\|_{\mathfrak{H}^{k}} = \left(\left\|A^{i}g\right\|^{2} + \left\|A^{i + k}g\right\|^{2}\right)^{1/2} \leq c\left(\alpha^{2i} + \alpha^{2(i + k)}\right)^{1/2} \leq \widetilde{c}\alpha^{i}, \quad \widetilde{c} = \left( 1 + \alpha^{2k}\right)^{1/2},
$$
is valid, the space $C_{\{1\}}(A)$ coincides with
$C_{\{1\}}^{k}(A)$. Moreover, the type of $g$ for the operator $A$ in
the space $C_{\{1\}}(A)$ is the same as the one for $A\upharpoonright
\mathfrak{H}^{k}$ in the space $C_{\{1\}}^{k}(A)$.

Denote by $\widetilde{A}$ the closure of $A$ in the space
$\mathfrak{H}^{-k}$. It is easy to make sure that $\widetilde{A}$ is a
nonnegative self-adjoint operator in $\mathfrak{H}^{-k}$ and if the
space $\mathfrak{H}$ is considered as a subspace of
$\mathfrak{H}^{-k}$, then we arrive at the previous situation. For
this reason,
$$
C_{\{m_{n}\}}(A) = C_{\{m_{n}\}}^{-k}(A), \quad C_{(m_{n})}(A) = C_{(m_{n})}^{-k}(A), \leqno (22)
$$
and
$$
\forall g \in C_{\{1\}}(A) = C_{\{1\}}^{k}(A): \sigma(g, A) = \sigma(g, \widetilde{A}).
$$
Taking into account that the restriction (extension) of the semigroup
$\left\{e^{-At}\right\}_{t \in \mathbb{R}_{+}}$ to the space
$\mathfrak{H}^{k_{1}}, \ k_{1} > 0,$ (to $\mathfrak{H}^{-k_{2}}, \
k_{2} > 0)$ is an analytic contractive $C_{0}$-semigroup in
$\mathfrak{H}^{k_{1}}$ (in $\mathfrak{H}^{-k_{2}}$), the embeddings
$C_{\{m_{n}\}}^{k}(A) \subseteq \mathfrak{B}, \ C_{(m_{n})}(A)
\subseteq \mathfrak{B}$ and the chain (19), we obtain for $y(t) =
e^{-At}f, \ f \in C_{\{m_{n}\}}(A)$ and $y_{0}(t) = e^{-At}g, \ g \in
C_{\{1\}}(A)$ that
$$
\left\|e^{-At}f - e^{-As}g\right\|_{\mathfrak{B}} \leq \left\|e^{-At}f - e^{-As}g\right\|_{\mathfrak{H}^{k_{1}}} \leq \|f - g\|_{\mathfrak{H}^{k_{1}}},
$$
whence
$$
\mathcal{E}_{r}(y, \mathfrak{B}) = \inf_{y_{0} \in S_{0}: \sigma (y_{0}) \leq r} \sup{s \in \mathbb{R}_{+}} \left\|e^{-As}f - e^{-As}g\right\|_{\mathfrak{B}} \leq \|f - g\|_{\mathfrak{H}^{k_{1}}},
$$
that is, \vspace*{-1mm}
$$
\mathcal{E}_{r}(y, \mathfrak{B}) \leq \mathcal{E}_{r}(f, A\upharpoonright \mathfrak{H}^{k_{1}}). \leqno (23)
$$

From (19) it follows that
$$
\forall t \in \mathbb{R}_{+}: \left\|e^{-At}f - e^{-At}g\right\|_{\mathfrak{H}^{-k_{2}}} \leq \left\|e^{-At}f - e^{-At}g\right\|_{\mathfrak{B}}.
$$
This implies that
$$
\|f - g\|_{\mathfrak{H}^{-k_{2}}} = \sup_{t \in \mathbb{R}_{+}} \left\|e^{-At}f - e^{-At}g\right\|_{\mathfrak{H}^{-k_{2}}} \leq \sup_{t \in \mathbb{R}_{+}} \left\|e^{-At}f - e^{-At}g\right\|_{\mathfrak{B}}
$$
and, hence,\vspace*{-2mm}
$$
\inf_{g \in C_{\{1\}}(A): \sigma (g) \leq r}\|f - g\|_{\mathfrak{H}^{-k_{2}}} \leq
\inf_{y_{0} \in S_{0}: \sigma (y_{0}) \leq r} \sup_{t \in \mathbb{R}_{+}} \|y(t) - y_{0}(t)\|_{\mathfrak{B}}.
$$
Thus,\vspace*{-1mm}
$$
\mathcal{E}_{r}(f, \widetilde{A}) \leq \mathcal{E}_{r}(y, \mathfrak{B}). \leqno (24)
$$
Inequalities (23) and (24), with regard to (19), (21), (22) and
Theorem 5, complete the proof of Theorem 7.
\end{proof}

\vspace*{-2mm}\medskip {\bf 5.} \ Let $A$ be a self-adjoint operator in
$\mathfrak{H}$ whose spectrum is discrete. Assume that its eigenvalues
$\lambda_{k} = \lambda_{k}(A), \ k \in \mathbb{N},$ satisfy the
condition $\sum_{k = 1}^{\infty} \lambda_{k}^{-p} < \infty$
with some $p > 0$. Suppose also that $\lambda_{k}$ are enumerated in
ascending order and each one is counted according to its multiplicity
and denote by $\{e_{n}\}_{n \in \mathbb{N}}$ the orthonormal basis in
$\mathfrak{H}$ consisting of eigenvectors of $A$. Then the spectral
function $E(\lambda)$ of the operator $A$ has the form
$$
E(\lambda)f = \sum_{\lambda_{k} \leq \lambda} f_{k}e_{k},
$$
where $f_{k} = (f_{k}, e_{k})$ are the Fourier coefficients of $f$, and
$$
\mathcal{E}_{r}(f, A) = \sum_{\lambda_{k} > r} f_{k}e_{k}.
$$
As it has been shown in [6], the following assertion holds true.

\begin{pr} The following equivalence relations are valid:
$$
\begin{array}{rcl} f \in C^{\infty}(A) & \Longleftrightarrow & \forall \alpha > 0, \ \exists c = c(\alpha) > 0:
|f_{k}| \leq c \lambda_{k}^{-\alpha}, \\
f \in C_{\{1\}}(A) & \Longleftrightarrow & \exists n_{0} \in \mathbb{N}: f_{k} = 0 \ \text{as} \ k \geq n_{0}, \\
f \in C_{\{m_{n}\}}(A) & \Longleftrightarrow & \exists \alpha > 0, \ \exists c > 0: |f_{k}| \leq c \tau^{-1}(\alpha
\lambda_{k}), \\
f \in C_{(m_{n})}(A) & \Longleftrightarrow & \forall \alpha > 0, \ \exists c = c(\alpha)> 0: |f_{k}|
\leq c \tau^{-1}(\alpha \lambda_{k}). \\
\end{array}
$$
(The function $\tau (\lambda)$ was defined in (11)).
\end{pr}

Let now $y(t)$ be a weak solution of (1). Then
$$
y(t) = \sum_{k = 1}^{\infty} e^{-\lambda_{k}t} f_{k}e_{k}, \quad \sum_{k = 1}^{\infty} |f_{k}|^{2} < \infty. \leqno (25)
$$
The solution $y(t)$ is an entire vector-valued function of exponential type $(y \in S_{0})$ if and only if
$$
\exists n_{0} \in \mathbb{N}: f_{k} = 0 \ \text{as} \ k \geq n_{0}.
$$
Proposition 1, Theorems 3, 5 and (13) imply the following.

\begin{theo}  The following equivalence relations take place:
$$
\begin{array}{rcl}
y \in C^{n}(\mathbb{R}_{+}, \mathfrak{H}) & \Longleftrightarrow & \mathcal{E}_{\lambda_{k}}(y)
= o\left(\lambda_{k + 1}^{-n}\right) \quad (n \to \infty), \\
y \in C^{\infty}(\mathbb{R}_{+}, \mathfrak{H}) & \Longleftrightarrow & \forall \alpha > 0:
\mathcal{E}_{\lambda_{k}}(y) = O \left(\lambda_{k + 1}^{-\alpha}\right) \quad (n \to \infty), \\
y \in C_{\{m_{n}\}} & \Longleftrightarrow & \exists \alpha > 0: \mathcal{E}_{\lambda_{k}}(y) =
O \left(\tau^{-1}(\alpha\lambda_{k + 1})\right) \quad (n \to \infty), \\
y \in C_{(m_{n})} & \Longleftrightarrow & \forall \alpha > 0: \mathcal{E}_{\lambda_{k}}(y) =
O \left(\tau^{-1}(\alpha\lambda_{k + 1})\right) \quad (n \to \infty). \\
\end{array}
$$
\end{theo}

\vspace*{-1mm}\medskip {\bf 6.}  Put $\mathfrak{H} = L_{2}(\Omega)$, where $\Omega$
is a bounded domain in $\mathbb{R}^{q}$ with piecewise smooth boundary
$\partial\Omega$, and denote by $B'$ the operator generated in
$L_{2}(\Omega)$ by the differential expression \vspace*{-3mm}
$$
(\mathcal{L}u)(x)
= - \sum_{i = 1}^{q}\sum_{k = 1}^{q}\frac{\partial}{\partial x_{i}}\left(a_{ik}(x)\frac{\partial u(x)}{\partial x_{k}}\right) + c(x)u(x), \leqno (26)
$$
on\vspace*{-1mm}
$$
\mathcal{D}(B') = \left\{u \in C^{2}(\overline{\Omega})\bigl|
  u\upharpoonright_{\partial\Omega} = 0\right\}. \leqno (27) $$ 
  It is
assumed that $a_{ik}(x), c(x) \in C^{\infty}(\overline{\Omega}), \
c(x) \geq 0$. Suppose also the expression (26) to be of elliptic type
in $\overline{\Omega}$. In this case all the eigenvalues $\mu_{i}(x),
\ i = 1, \dots q$, of the matrix $\|a_{ik}(x)\|_{i,k = 1}^{q}, \ x \in
\overline{\Omega}$, have the same sign; without loss of generality we
may assume $\mu_{i}(x) > 0, \ x \in \overline{\Omega}$.

It is not hard to make sure that $B'$ is a positive definite Hermitian
operator with dense domain in $L_{2}(\Omega)$.  So, $B'$ admits a
closure to a positive definite selfadjoint operator $B$ on
$L_{2}(\Omega)$.  We shall call $B$ the operator generated by (26),
(27). The spectrum of $B$ is discrete, and for its eigenvalues,
$\lambda_{1}(B) < \lambda_{2}(B) < \dots < \lambda_{n}(B) < \dots$, the
estimate
$$
c_{1}n^{2/q} \leq \lambda_{n}(B) \leq c_{2}n^{2/q}, \quad 0 < c_{i} = {\rm {const}}, \quad i = 1, 2. \leqno (28)
$$
is valid (see [8]). Denote by $e_{n}(x), \ n \in \mathbb{N}$, the
orthonormal basis in $L_{2}(\Omega)$, consisting of eigenfunctions of
$B$.

In the case where $\Omega$ is a $q$-dimensional cube, $0 < x_{k} < a, \
k = 1, \dots, q, \ a > 0$, and $\mathcal{L} = -\sum_{i =
  1}^{q}\frac{\partial^{2}}{\partial x_{i}^{2}}$, the following
formulas for the eigenvalues $\lambda_{n_{1}\dots n_{q}}, \ n_{k} \in
\mathbb{N}$, and eigenfunctions $e_{n_{1}\dots n_{q}}(x)$ of the
operator $B$ hold: \vspace*{-2mm}
$$
\lambda_{n_{1}\dots n_{q}} = \frac{\pi^{2}}{a^{2}}\sum_{k = 1}^{q}n_{k}^{2}; \quad e_{n_{1}\dots n_{q}}(x) = \left( \frac{2}{a}\right)^{q/2}\prod_{k = 1}^{q} \sin\frac{\pi}{a}x_{k}.
$$

Let $y(t) = u(t, x) \in C(\mathbb{R}_{+}, L_{2}(\Omega))$ be a weak
solution of the problem
$$
\left(\frac{\partial}{\partial t} - \sum_{k = 1}^{q}\sum_{i = 1}^{q} \frac{\partial}{\partial x_{k}}\left(a_{ki}(x)\frac{\partial}{\partial x_{i}}\right) + c(x)\right)u(t,x) = 0, \leqno (29)
$$
$$
\forall t > 0, \ \forall x \in \partial \Omega: u(t, x) = 0,  \leqno (30)
$$
where the conditions on $a_{ki}(x)$ and $c(x)$ are the same as
before. Then $u(t, x)$ admits a representation of form (25). Set
$$
L_{2}^{n} = \mathcal{D}(B^{n}), \quad \|f\|_{L_{2}^{n}} = \left( \|f\|^{2} + \left\|B^{n}f\right\|^{2}\right)^{1/2},
$$
where $B$ is an operator generated in the space $L_{2}(\Omega)$ by
expression (26) and boundary value condition (27). The space
$L_{2}^{n}$ is continuously and densely embedded into
$L_{2}(\Omega)$. Denote by $L_{2}^{-n}$ the negative space
corresponding to the positive one $L_{2}^{n} \subset
L_{2}(\Omega)$. In the case where $\mathcal{L} = -\sum_{k =
  1}^{q} \frac{\partial^{2}}{\partial x_{k}^{2}}$, \ $L_{2}^{n}$ is
none other than the well-known Sobolev space
$\stackrel{\circ}{W}_{2}^{2n}(\Omega)$.

Using estimate (28) for $\lambda_{k}(B)$ and Theorem 8, we obtain in a
way analogous to that used in the proof of Theorem 7 the following
assertion.

\begin{theo} Let $\mathfrak{B}$ be a Banach space and let
$$
L_{2}^{n_{1}} \subseteq \mathfrak{B} \subseteq L_{2}^{-n_{2}}, \quad n_{1}, n_{2} \in \mathbb{N},
$$
be a chain of continuously and densely embedded into each other
spaces. Suppose also the sequence $\left\{m_{n}\right\}_{n =
  1}^{\infty}$ to satisfy (10) and (12). Then
$$
\begin{array}{rcl}
y(t) = u(t, x) \in C^{\infty}(\mathbb{R}_{+}, L_{2}(\Omega)) & \Longleftrightarrow & \alpha > 0:\mathcal{E}_{\lambda_{k}}^{\mathfrak{B}}(y) = O\left(\frac{1}{(k + 1)^{\alpha}}\right), \\
y(t) \in C_{\{m_{n}\}} & \Longleftrightarrow & \exists \alpha > 0: \mathcal{E}_{\lambda_{k}}^{\mathfrak{B}}(y) = O \left(\tau^{-1}\left(\alpha(k + 1)^{2/q}\right)\right), \\
y(t) \in C_{(m_{n}} & \Longleftrightarrow & \forall \alpha > 0: \mathcal{E}_{\lambda_{k}}^{\mathfrak{B}}(y) = O \left(\tau^{-1}\left(\alpha(k + 1)^{2/q}\right)\right), \\
\end{array}
$$
where
$$
\mathcal{E}_{\lambda_{k}}^{\mathfrak{B}}(y) = \inf_{y_{0} \in S_{0}: \sigma(y_{0})\leq \lambda_{k}} \sup_{s \in \mathbb{R}_{+}} \|y(s) - y_{0}(s)\|_{\mathfrak{B}}.
$$
\end{theo}

It is relevant to remark that in the case where $a_{ki}(x) =
\delta_{ki}$ and $c(x) \equiv 0$, by virtue of the embedding theorems
for Sobolev spaces, one can take the space $C(\overline{\Omega})$ of
continuous in $\overline{\Omega}$ functions or $L_{p}(\Omega), \ 1
\leq p < \infty,$ as $\mathfrak{B}$ and consider not only the
Dirichlet but some other boundary value problems, in particular, the
Neumann problem.

\end{document}